\documentclass[11pt,oneside]{amsart}
\usepackage{t1enc}
\usepackage[latin2]{inputenc}

\usepackage{amsmath,amsfonts}
\usepackage{amsthm}
\usepackage{amscd}
\usepackage{amssymb}
\usepackage{enumerate}
\usepackage{mathrsfs}
\usepackage{dsfont}

\usepackage{stmaryrd}
\usepackage[T1]{fontenc}

\usepackage{mathtools}

\usepackage{bbm}

\usepackage{tikz}
\usepackage{tikz-cd}
\usepackage{wasysym}
\usepackage{verbatim}
\usepackage{fancyhdr}
\usepackage{fancybox}
\usepackage{url}
\usepackage{color}

\newcommand{\eps}{\varepsilon}

\newcommand{\bc}{\begin{center}}
\newcommand{\ec}{\end{center}}

\DeclareMathOperator{\supp}{supp}

\newtheorem{thm}{Theorem}[section]
\newtheorem{lem}[thm]{Lemma}
\newtheorem{prop}[thm]{Proposition}
\newtheorem{cor}[thm]{Corollary}

\theoremstyle{definition}

\newtheorem{df}[thm]{Definition}

\newtheorem{rem}[thm]{Remark}

\title[There is a P-measure in the random model]{There is a P-measure in the random model}

\author{Piotr Borodulin--Nadzieja}
\address[Piotr Borodulin-Nadzieja]{Mathematical Institute, University of Wroc\l aw, \ \ \  pl. Grunwaldzki 2, 50-384 Wroc\l aw, Poland}
\email{pborod@math.uni.wroc.pl}

\author{Damian Sobota}
\address[Damian Sobota]{Kurt G\"odel Research Center for Mathematical Logic, Department of Mathematics, University of Vienna, Kolingasse 14-16, 1090 Vienna, Austria}
\email{damian.sobota@univie.ac.at}

\thanks{The first author was supported by the 
 National Science Center project no. 2018/29/B/ST1/00223. The second author was supported by the Austrian Science Fund (FWF) grant no. 4570-N}


\begin{document}

\maketitle

\begin{abstract}
We say that a finitely additive probability measure $\mu$ on $\omega$ is \emph{a P-measure} if it vanishes on points and for each decreasing sequence $(E_n)$ of infinite subsets of $\omega$ there is $E\subseteq\omega$ such that $E\subseteq^* E_n$ for each $n\in\omega$ and $\mu(E) = \lim_{n\to\infty}\mu(E_n)$. Thus, P-measures generalize in a natural way  P-points and it is known that, similarly as in the case of P-points, their existence is independent of $\mathsf{ZFC}$. In this paper we show that there is a P-measure in the model obtained by adding any number of random reals to a model of $\mathsf{CH}$. As a corollary, we obtain that in the classical random model $\omega^*$ contains a nowhere dense ccc closed P-set.
\end{abstract}

\section{Introduction}

We say that a (finitely additive) probability measure $\mu$ on $\omega$ is \emph{a P-measure} if it vanishes on points and for each decreasing sequence $(E_n)$ of infinite subsets of $\omega$ there is $E\subseteq\omega$ such that $E\subseteq^* E_n$ for each $n\in\omega$ and
$\mu(E) = \lim_{n\to\infty}\mu(E_n)$ (see Section \ref{sec:prelim} for the unexplained terminology).

The importance of this notion lies in the fact that it generalizes the notion of $P$-points in $\omega^*$ in a natural way. Indeed, consider an ultrafilter $\mathcal{U}$ on $\omega$ and let $\mu\colon\wp(\omega)\to\{0,1\}$ be the one-point measure concentrated in $\mathcal{U}$, i.e., for every $A\in\wp(\omega)$, $\mu(A)
= 1$ if and only if $A\in \mathcal{U}$. Then, $\mu$ is a P-measure if and only if $\mathcal{U}$ is a P-point. 
Moreover, if $\mathcal{U}$ is a P-point, then the measure $\nu$ on $\omega$,  defined for every $A\in\wp(\omega)$ by the formula
\[ \nu(A) = \lim_{n\to \mathcal{U}} \frac{|A \cap n|}{n}\]
is a non-atomic P-measure (extending the asymptotic density). So, every P-point \emph{is} a P-measure and, moreover, it induces a P-measure which additionally is non-atomic. 

The converse issue however, that is, whether the existence of P-measures implies the existence of P-points, is so far unresolved. It was considered by Mekler in \cite{Mekler} (where the P-measures were introduced for the first time, under the name of \emph{measures with the additive
property}, or \emph{with (AP)} in short), by
Blass \emph{et al.} in \cite{PlebanekRyll} (where P-measures are called \emph{measures with AP(*)}) and by Greb\'\i k in \cite{Grebik} (where P-measures are also called \emph{measures with AP(*)}). Unfortunately, all the constructions of P-measures presented in those articles involve more or less directly
P-points.

Recall that basically only two models of $\mathsf{ZFC}$ without P-points are known: the classical construction due to Shelah and its variants (see \cite{Shelah-proper}, cf. also \cite{FSZ} and \cite{FGZ}), and the Silver model (see \cite{David-Osvaldo}). 
Mekler \cite{Mekler} showed than in the model used by Shelah there are no P-measures. It is unknown if there is a P-measure in the
Silver model.

Also, it is still open if there is a P-point in the classical random model, i.e. the model obtained by adding $\omega_2$ random reals to a model satisfying the Continuum Hypothesis ($\mathsf{CH}$). Although in \cite{Cohen} the author claimed to answer this problem in positive, it has recently appeared that his proof contains an essential gap (see \cite{David-Osvaldo}). Recall that Kunen proved that if we first add $\omega_1$ Cohen reals to the ground model and then $\omega_2$ random reals, then we get a model with a P-point (see \cite{Jorg-problems}). Recently, Dow \cite{Dow-Pfilters} also
showed that if we add $\omega_2$ random reals to a model of $\mathsf{CH}+\square_{\omega_1}$, then the resulting model contains a P-point.

The main purpose of this paper is to show that in the classical random model there is a P-measure (Corollary \ref{cor:random_ap}). Our measure is not defined using any P-point from the extension (we still do not know if there exists one!) and it seems to be the first example of this kind. The construction exploits some ideas of Solovay from \cite{Solovay} (and so we call the resulting measure \emph{a Solovay measure}). It proceeds more or less as follows. We start with an ultrafilter $\mathcal{U}$ in the ground model $V$ and we generate, in a certain way, a name $\dot{\mu}_\mathcal{U}$ for a finitely additive measure on $\omega$. In the random extension $V[G]$, the resulting measure $(\dot{\mu}_\mathcal{U})_G$ in a sense extends the ground model one-point measure $\delta_\mathcal{U}$ concentrated in $\mathcal{U}$, and we show that whenever $\mathcal{U}$ is a ground model P-point, then $(\dot{\mu}_\mathcal{U})_G$ is a P-measure in $V[G]$ which do not extend the asymptotic density (Theorem \ref{thm:random_pp_ap} and Proposition \ref{rho_extends}). So,
perversely, we do use a P-point in the construction of our P-measure, but this P-point lives in another model.


We show that all Solovay measures are non-atomic (Theorem \ref{non-atomic}). Since the supports of non-atomic measures cannot contain P-points (cf. Proposition \ref{Ppoint-measure}), the same is true for Solovay measures. But, as the support of a P-measure is a nowhere dense ccc closed P-subset of $\omega^*$, we get that in the random model (or in any random extension of a model of $\mathsf{CH}$) there is a nowhere dense ccc closed P-set in $\omega^*$ (Corollary \ref{cor:random_ap2}). This is an interesting result as there are models without such subsets of $\omega^*$, see e.g. \cite{FSZ}.

It turns out however that supports of Solovay measures in random extensions can be used to obtain characterizations of various classes of ultrafilters on $\omega$ in the ground model. Notice that every ultrafilter $\mathcal{U}$ on $\omega$ from the ground model $V$ generates a closed subset $K_\mathcal{U}$ of $\beta\omega$ in the extension $V[G]$ (consisting of all ultrafilters extending the filter $\mathcal{U}$). We show that an ultrafilter $\mathcal{U}$ is semi-selective (in $V$) if and only if the measure $(\dot{\mu}_\mathcal{U})_G$ is strictly positive on $K_\mathcal{U}$, and so its support is the whole $K_\mathcal{U}$ (Theorem \ref{semi-selective}). This result sheds some new light on the well-known theorem, due to Kunen \cite{Kunen-special-points}, saying that no P-point in the random model can extend a ground model semi-selective ultrafilter. Indeed, as we already said, no P-point can be an element of the support of a non-atomic measure, so, if $\mathcal{U}$ is semi-selective, then no element of $K_\mathcal{U}$ can be a P-point (see Corollary \ref{cor:kunen}).

In the context of aforementioned Kunen's theorem it is natural to ask whether a P-point $\mathcal{V}$ in the random extension (assuming it exists) can at all extend a ground model ultrafilter $\mathcal{U}$. If this is the case, then it is easy to observe that $\mathcal{U}$ itself must be a P-point. However, by Kunen's result $\mathcal{U}$ cannot be semi-selective and so, as explained in the beginning of Section \ref{sec:canjar}, there exists a decreasing sequence $(E_n)$ of infinite subsets of $\omega$ (in the extension) such that $\mathcal{U}$ and $\{E_n\colon n\in\omega\}$ together generate a filter which cannot be further extended to a bigger filter by \emph{any} pseudointersection of the sequence $(E_n)$. On the other hand, 
if $\mathcal{U}$ is a Canjar ultrafilter, then 
for every decreasing sequence $(E_n)$ of infinite subsets of $\omega$, if $\mathcal{F} = \mathcal{U} \cup \{E_n\colon n\in \omega\}$ is a base of a filter, then there is a pseudointersection $E$ of $(E_n)$ such $\mathcal{F} \cup \{E\}$ is still a base of a filter (Proposition \ref{lem:canjar}). Hence, if there are ultrafilters in the ground model which can be extended to a P-point in the random extension, then Canjar ultrafilters should be considered as promising candidates. 

\section{Preliminaries\label{sec:prelim}}

$\wp(\omega)$ denotes the power set of $\omega$, that is, the family of all subsets of $\omega$. If $A,B\in\wp(\omega)$, then we write $A\subseteq^*B$ if the set $A\setminus B$ is finite, and $A=^*B$ if $A\subseteq^*B$ and $B\subseteq^*A$. \emph{An
antichain} in $\wp(\omega)$ is a sequence $(A_n)$ of pairwise disjoint subsets of $\omega$. If $(A_n)$ is a collection of subsets of $\omega$, then if $A_n \subseteq^* A$ (resp. $A\subseteq^* A_n$) for each $n\in\omega$, then $A$ is called \emph{a
pseudounion} (resp. \emph{a pseudointersection}) of $(A_n)$. Note that if $\bigcup_{n\in\omega} A_n\subseteq A$, then $A$ is a pseudounion of $(A_n)$ and hence $\omega$ is trivially a pseudounion of every family.

$Fin$ denotes the ideal of all finite subsets of $\omega$. Also, for every $A\in\wp(\omega)$ by $[A]^{<\omega}$ we denote the set of all finite subsets of $A$ and by $[A]^\omega$ the set of all infinite subsets of $A$.

If $\mathbb{A}$ is a Boolean algebra, then by $St(\mathbb{A})$ we denote its Stone space. Recall that $\mathbb{A}$ is isomorphic to the algebra $Clopen(St(\mathbb{A}))$ of all clopen subsets of $St(\mathbb{A})$; for every $A\in\mathbb{A}$ by
$[A]_\mathbb{A}$ we will denote the corresponding clopen subset of $St(\mathbb{A})$. Recall also that the Stone spaces $St(\wp(\omega))$ and $St(\wp(\omega)/Fin)$ are homeomorphic to $\beta\omega$, the \v{C}ech--Stone compactification of $\omega$,
and $\omega^*=\beta\omega\setminus\omega$, respectively. For every $A\in\wp(\omega)$ we will simply write $[A]$ for $[A]_{\wp(\omega)}$.

Let $X$ be a topological space. We say that $X$ is \emph{an almost P-space} if every non-empty $\mathbb{G}_\delta$ subspace of $X$ has non-empty interior. A subset $Y$ of $X$ is \emph{a P-set} if the intersection of countably many open sets containing $Y$ contains $Y$ in its interior. A point $x\in X$ is \emph{a P-point} in $X$ if the singleton $\{x\}$ is a P-set in $X$. If $X=\omega^*$, this definition of course coincides with the standard one of a P-point. 

All \emph{measures} considered on Boolean algebras are assumed to be non-negative, finitely additive, and finite (i.e. if $\mu$ is a measure on a Boolean algebra $\mathbb{A}$, then $\mu(1_\mathbb{A})<\infty$). A measure $\mu$ on a Boolean algebra $\mathbb{A}$ is \emph{probability} if $\mu(1_\mathbb{A})=1$. The space of all probability measures on $\mathbb{A}$ will be denoted by $P(\mathbb{A})$. All \emph{measures} on compact spaces are meant to be non-negative, $\sigma$-additive, and \emph{Radon}, that is, Borel, regular, and finite (i.e. if $\mu$ is a measure on a compact space $K$, then $\mu(K)<\infty$). Similarly as above, a measure $\mu$ on a compact space $K$ is \emph{probability} if $\mu(K)=1$. The space of all probability measures on $K$ will be denoted by $P(K)$. \emph{The support} of a measure $\mu$ on a compact space $K$ is denoted by $\supp(\mu)$.

If $\mu$ is a measure on the Boolean algebra $\wp(\omega)$, then we simply say that \emph{$\mu$ is a measure on $\omega$}. If $\mu$ is a measure on $\omega$ and $\mu(\{n\})=0$ for every $n\in\omega$, then we also say that \emph{$\mu$ vanishes on points}.

Recall that each measure $\mu$ on a Boolean algebra $\mathbb{A}$ extends to a unique measure $\check{\mu}$ on the Stone space $St(\mathbb{A})$ such that $\mu(A)=\check{\mu}([A]_\mathbb{A})$ for every $A\in\mathbb{A}$. Also, for every measure $\nu$ on
a totally disconnected compact space $K$, there is a measure $\mu$ on the Boolean algebra of clopen subsets of $K$ such that $\nu = \check{\mu}$. In particular, every measure on $\beta\omega$ is induced by a measure on $\omega$ and every measure on $\omega^*$
is induced by a measure on $\wp(\omega)/Fin$, or, equivalently, by a measure on $\omega$ which vanishes on points.

\section{Basic facts on P-measures}

In this section we will collect several general facts on P-measures (some of them are standard and folklore), which will be useful in the sequel. We start with the following simple lemma.

\begin{lem}\label{lem:basic}
\ 
\begin{enumerate}
	\item If $\mu$ is a measure on $\omega$ vanishing on points, then for every $A\in[\omega]^\omega$ there is $B\in[A]^\omega$ such that $\mu(B)=0$.

	\item The support of a measure on $\omega^*$ vanishing on points is nowhere dense in $\omega^*$.

	\item There is no strictly positive measure on $\omega^*$, that is, a measure $\mu$ on $\omega^*$ such that $\mu(U)>0$ for every clopen subset $U$ of $\omega^*$.
\end{enumerate}
\end{lem}
\begin{proof}
	(1) Let $\mu$ be a measure on $\omega$ which vanishes on points. Fix $A\in[\omega]^\omega$ and let $\mathcal{A}\subseteq[A]^\omega$ be an uncountable family of pairwise almost disjoint infinite subsets of $A$. If $B, C\in \mathcal{A}$ are distinct, then $\mu(B\cup C) = \mu(B)+\mu(C)$, and so 
	there exists $D\in\mathcal{A}$ such that $\mu(D)=0$.

	(2) and (3) follow immediately from (1).
\end{proof}

Clearly, measures on $\omega$ which vanish on points cannot be $\sigma$-additive. However, P-measures are 'almost' $\sigma$-additive as the following proposition shows. We leave the proof of this standard fact to the reader. 

\begin{prop}\label{prop:ap_cont}
	Let $\mu$ be a measure on $\omega$ vanishing on points. Then, the following are equivalent:
	\begin{enumerate}
		\item $\mu$ is a P-measure, that is, for every decreasing sequence $(A_n)$ in $\wp(\omega)$ there is a pseudointersection $A\in\wp(\omega)$ of $(A_n)$ such that $\mu(A)=\lim_{n\to\infty}\mu(A_n)$;
		\item for every increasing sequence $(A_n)$ in $\wp(\omega)$ there is a pseudounion $A\in\wp(\omega)$ of $(A_n)$ such that $\mu(A)=\lim_{n\to\infty}\mu(A_n)$;
		\item  for every antichain $(A_n)$ in $\wp(\omega)$ there is a pseudounion $A\in\wp(\omega)$ of $(A_n)$ such that $\mu(A)=\sum_{n\in\omega}\mu(A_n)$.
	\end{enumerate}
\end{prop}

\begin{lem}\label{epsilon} Assume $\mu$ is not a P-measure. Then, there is a decreasing sequence $(A_n)$ of subsets of $\omega$ and $\varepsilon>0$ such that $\mu(A) < \lim_{n\to\infty} \mu(A_n) - \varepsilon$ for every pseudointersection $A$ of $(A_n)$. 
	
\end{lem}

\begin{proof} 
Let $(A_n)$ be a decreasing sequence witnessing that $\mu$ is not a P-measure. Suppose for the contradiction that for every $k\in\omega$ there is a pseudointersection $B_k$ of $(A_n)$ such that $\mu(B_k)\geq \lim_{n\to\infty} \mu(A_n) - 1/k$. Then the sequence $(B_0 \cup \dots \cup B_n, A_n)$ forms an $(\omega,\omega)$-pregap. Since there
	is no $(\omega,\omega)$-gap (see e.g. \cite{Scheepers}), it can be interpolated, i.e. there is a set $B\subseteq \omega$ such that $B_n \subseteq^* B \subseteq^* A_n$ for each $n\in\omega$. Clearly, $B$ is a pseudointersection of $(A_n)$ and $\mu(B) =
	\lim_{n\to\infty} \mu(A_n)$, a contradiction.
\end{proof}

Recall that a finitely additive measure $\mu$ on a Boolean algebra $\mathbb{A}$ is \emph{non-atomic} if for each $\varepsilon>0$ there is a finite partition of the unit element $1_\mathbb{A}$ into elements of $\mu$-measure $\le\varepsilon$ (note
that in some sources such measures are called 'strongly non-atomic', see e.g.  \cite{RaoRao}). Note that a measure $\mu$ on $\mathbb{A}$ is non-atomic if and only if its Radon extension $\check{\mu}$ onto the Stone space $St(\mathbb{A})$ vanishes on
points (that is, $\check{\mu}(\{x\})=0$ for every $x\in St(\mathbb{A})$). A classical theorem of Sierpi\'nski \cite{Sie22} asserts that every non-atomic $\sigma$-additive measure $\mu$ on a $\sigma$-complete Boolean algebra $\mathbb{A}$ attains every value
from the interval $[0,\mu(1_\mathbb{A})]$ (or, more precisely, for every $A\in\mathbb{A}$ and $\alpha\in[0,\mu(A)]$ there is $B\subseteq A$ in $\mathbb{A}$ such that $\mu(B)=\alpha$). It appears however that the proof of Sierpi\'nski's theorem
presented in  \cite[Theorem 5.1.6]{RaoRao} also works \emph{mutatis mutandis} for P-measures on $\omega$.

\begin{cor}\label{nonatomic_ap}
Let $\mu$ be a non-atomic P-measure on $\omega$. Then, for every $A\in\wp(\omega)$ and $\alpha\in[0,\mu(A)]$ there is $B\subseteq A$ such that $\mu(B)=\alpha$. 
\end{cor}

The next proposition shows that atoms of P-measures give rise to P-points in $\omega^*$.

\begin{prop}\label{prop:pmeas_supp_pp}
Let $\mu$ be a P-measure on $\omega$. If there is $x\in\omega^*$ such that $\check{\mu}(\{x\})>0$, then $x$ is a P-point.
\end{prop}
\begin{proof}
Assume that $x\in\omega^*$ is such that $\delta=\check{\mu}(\{x\})>0$. There exists $A\in\wp(\omega)$ such that $x\in[A]$ and $\check{\mu}([A]\setminus\{x\})<\delta$. Let $(A_n)$ be a sequence of elements of $x$ decreasing to $\emptyset$. Since $\mu$ is a P-measure, there is $B\in\wp(\omega)$ such that $B\subseteq^* A_n\cap A$ for every $n\in\omega$ and $\mu(B)=\lim_{n\to\infty}\mu(A_n\cap A)$. As $x$ is an element of each clopen $[A_n\cap A]$, we have $\mu(B)\ge\delta$. It follows that $x\in[B]$. Indeed, if $x\not\in[B]$, then, since $B\subseteq^*A$, it holds $\delta\le\mu(B)=\check{\mu}([B])\le\check{\mu}([A]\setminus\{x\})<\delta$, a contradiction. So, $B\in x$ and $B\subseteq^* A_n$ for each $n\in\omega$. It follows that $x$ is a P-point.
\end{proof}

\begin{cor}
If there is no P-point, then every P-measure is non-atomic.
\end{cor}

The following fact, which we will need later, is folklore, but we include its proof for the sake of self-containment of the paper.

\begin{prop}\label{Ppoint-measure} Suppose that $\mu$ is a non-atomic measure on $\beta\omega$. If $\mathcal{U}$ is a P-point on $\omega$, then $\mathcal{U}$ does not belong to the support of $\mu$.
\end{prop}

\begin{proof} The non-atomicity of $\mu$ implies that there is a family 
	$\{N_\tau\colon \tau\in 2^{<\omega}\} \subseteq \mathcal{P}(\omega)$ such that 
	\begin{itemize}
		\item $\mu([N_\tau])>0$ for each $\tau\in 2^{<\omega}$, 
		\item $N_\tau$ is a disjoint union of  $N_{\tau^\smallfrown 0}$ and  $N_{\tau ^\smallfrown 1}$ for each $\tau \in 2^{<\omega}$, and
		\item $\lim_{n\to \infty} \mu([N_{s \restriction n}]) = 0$ for each $s\in 2^\omega$.
	\end{itemize}
	For each $n\in\omega$ there is $\tau_n \in 2^n$ such that $N_{\tau_n} \in \mathcal{U}$. Clearly, $\{\tau_n\colon n\in \omega\}$ forms a branch in the tree $2^{<\omega}$. If $\mathcal{U}$ is a
	P-point, then there is $N \in \mathcal{U}$ such that $N \subseteq^* N_{\tau_n}$ for each $n\in\omega$. Since $\mu$ is non-atomic, we have $\mu([N])=0$ and so the clopen $[N]$ witnesses that $\mathcal{U}$ lies outside the support of $\mu$.
\end{proof}

\section{Measures on $\omega$ in random extensions}

Fix a cardinal number $\kappa\geq \omega$. Let $\mathrm{Bor}(2^\kappa)$ denote the $\sigma$-field of all Borel subsets of the space $2^\kappa$. By $\lambda_\kappa$ denote the standard product measure on $2^\kappa$ and set $\mathcal{N}_\kappa = \{B\in \mathrm{Bor}(2^\kappa)\colon \lambda_\kappa(B) = 0\}$.
Let $\mathbb{M}_\kappa$ be the measure algebra of type $\kappa$, i.e. $\mathbb{M}_\kappa = \mathrm{Bor}(2^\kappa)/{\mathcal{N}_\kappa}$ with the standard Boolean operations $\wedge$, $\vee$, $^c$, zero $0$ and unit $1$ elements, and ordering $\le$.
The measure $\lambda_\kappa$ induces in the natural way the measure on $\mathbb{M}_\kappa$ which we will also denote by $\lambda_\kappa$. 

By $V$ we denote the set-theoretic universe. Speaking about \emph{forcing with $\mathbb{M}_\kappa$}, we formally mean the partial ordering $(\mathbb{M}_\kappa\setminus\{0\},\le)$---this abuse of notation should not lead to any confusion. When forcing with $\mathbb{M}_\kappa$ (i.e. \emph{adding $\kappa$ random reals}), if $G$ is a generic filter over $V$, then for every $\mathbb{M}_\kappa$-name $\sigma$ by $\sigma_G$ we denote its interpretation through $G$ in the extended model $V[G]$. Note that $\mathbb{M}_\kappa$ is forcing equivalent to the poset $(\mathrm{Bor}(2^\kappa)\setminus \mathcal{N}_\kappa, \subseteq)$.

\subsection{Names for subsets of $\omega$}\label{sec:names-subsets}

In this section we will present a handy way of dealing with names for subsets of $\omega$ in random extensions.

Consider the product Boolean algebra $\mathbb{M}_\kappa^\omega$ with the Boolean operations defined coordinate-wise, i.e. for every $M,N\in\mathbb{M}_\kappa^\omega$ and $k\in\omega$ we set:
\begin{itemize}
	\item $(M \wedge N)(k) = M(k) \wedge N(k)$,
	\item $(M \vee N)(k) = M(k) \vee N(k)$,
	\item $M^c(k) = M(k)^c$.
\end{itemize}
As above, the zero and unit elements of $\mathbb{M}_\kappa^\omega$ will be denoted by $0$ and $1$, respectively. On $\mathbb{M}_\kappa^\omega$ we consider two orderings: the standard one $\le$ originating from $\mathbb{M}_\kappa$, and the 'almost' one $\le^*$, defined for every $M,N\in\mathbb{M}_\kappa^\omega$ as follows: $M\le^*N$ if $M(k)\le N(k)$ for almost all $k\in\omega$.

For $M\in \mathbb{M}_\kappa^\omega$ define an $\mathbb{M}_\kappa$-name $\dot{M}$ by the formula:
\[\dot{M}=\{\langle k, M(k)\rangle\colon k\in \omega\}.\]
Clearly, $\dot{M}$ is \emph{an $\mathbb{M}_\kappa$-name for a subset of $\omega$}, that is, $\Vdash_{\mathbb{M}_\kappa}\dot{M}\subseteq\omega$. The proof of the next lemma is left to the reader.

\begin{lem}\label{homodot} For each $M,N\in\mathbb{M}_\kappa^\omega$ we have $\Vdash_{\mathbb{M}_\kappa}\dot{M}\cap\dot{N}=(M\wedge N)\dot{\ }$, and $\Vdash_{\mathbb{M}_\kappa}\dot{M}^c=(M^c)\dot{\ }$.
\end{lem}

\begin{lem}\label{names_conditions}
For every $M\in\mathbb{M}_\kappa^\omega$ and every $k\in\omega$ we have $M(k)=\llbracket k\in\dot{M}\rrbracket$.
\end{lem}
\begin{proof}
Fix $M\in\mathbb{M}_\kappa^\omega$ and $k\in\omega$. If $M(k)\in\{0,1\}$, then we are trivially done, so assume the opposite. Since $M(k)\Vdash k\in\dot{M}$, $M(k)\le \llbracket k\in\dot{M}\rrbracket$. Similarly, by Lemma \ref{homodot}, $M(k)^c\le\llbracket k\in(M^c)\dot{\ }\rrbracket=\llbracket k\in\dot{M}^c\rrbracket=\llbracket k\not\in\dot{M}\rrbracket=\llbracket k\in\dot{M}\rrbracket^c$. It follows thus that $M(k)=\llbracket k\in\dot{M}\rrbracket$.
\end{proof}

We will use Lemma \ref{names_conditions} frequently without direct referring to it.

\begin{cor}\label{homodotgm}
For every $M,N\in\mathbb{M}_\kappa^\omega$ it holds:
\begin{enumerate}
	\item if $M\le N$, then $\Vdash_{\mathbb{M}_\kappa}\dot{M}\subseteq\dot{N}$;
	\item if $M\wedge N=0$, then $\Vdash_{\mathbb{M}_\kappa}\dot{M}\cap\dot{N}=\emptyset$. 
\end{enumerate}
\end{cor}

Actually, it appears that each subset of $\omega$ in $V^{\mathbb{M}_\kappa}$ can be described in the above way.

\begin{lem}\label{names} For each $\mathbb{M}_\kappa$-name $\dot{N}$ there is $M\in \mathbb{M}_\kappa^\omega$ such that
\[\Vdash_{\mathbb{M}_\kappa} (\dot{N}\subseteq\omega\ \Rightarrow\ \dot{N} = \dot{M}).\]
\end{lem}
\begin{proof}
For $k\in \omega$ let $M(k) = \llbracket k\in \dot{N} \rrbracket$. 
\end{proof}

So, in what follows, we will assume that all names for subsets of $\omega$ are basically elements of $\mathbb{M}_\kappa^\omega$. In particular, when writing that \emph{$\dot{N}$ is an $\mathbb{M}_\kappa$-name for a subset of $\omega$}, we will assume that $\dot{N}$ comes from the element $N\in\mathbb{M}_\kappa^\omega$ constructed as in the proof of the above lemma. 

\begin{lem}\label{almost} For every $M, N\in \mathbb{M}_\kappa^\omega$ we have:
\begin{itemize}
	\item $\bigwedge_{l\in\omega} \bigvee_{k>l} M(k) = 0$ if and only if $\Vdash_{\mathbb{M}_\kappa}$``$\dot{M}$ is finite'';
	\item $\bigwedge_{l\in\omega} \bigvee_{k>l} (M(k)\setminus N(k)) = 0$ if and only if $\Vdash_{\mathbb{M}_\kappa} \dot{M} \subseteq^* \dot{N}$;
	\item if $M\le^* N$, then $\Vdash_{\mathbb{M}_\kappa} \dot{M} \subseteq^* \dot{N}$.
\end{itemize}
\end{lem}
\begin{proof} Just notice that $\llbracket \forall l \ \exists k>l \ k\in \dot{M}\rrbracket = \bigwedge_{l\in\omega} \bigvee_{k>l} M(k)$.
\end{proof}

We need the following piece of notation in the sequel.

\begin{df}\label{notacja} For $X\subseteq \omega$ define $\check{X} \in \mathbb{M}^\omega_\kappa$ for every $k\in\omega$ by
$$
\check{X}(k)=
\begin{cases}
	1, \mbox{ if } k\in X, \\
	0, \mbox{ otherwise.}
\end{cases}
$$
For $q\in \mathbb{M}_\kappa$ define $\vec{q} \in \mathbb{M}^\omega_\kappa$ for every $k\in\omega$ by
\[ \vec{q}(k) = q. \]
\end{df}

Notice that if we interpret $\check{X}$ as an $\mathbb{M}_\kappa$-name $(\check{X})\dot{ }$ for a subset of $\omega$, then it is just the canonical $\mathbb{M}_\kappa$-name for $X$ (so the usage of $\check{\ }$ is justified). We will thus also write $\check{X}$ for the interpretation of $\check{X}$ as an $\mathbb{M}_\kappa$-name, omitting the dot $\dot{ }$.

\begin{lem}\label{restr_incl}
Let $M,N\in\mathbb{M}_\kappa^\omega$ and $q\in\mathbb{M}_\kappa$ be such that $q\Vdash\dot{M}\subseteq\dot{N}$. Then, for $M'=M\wedge\vec{q}$ and $N'=N\wedge\vec{q}$ we have:
\begin{enumerate}
	\item $M'\le N'$, and
	\item $q\Vdash\dot{M}=\dot{M}',\dot{N}=\dot{N}'$.
\end{enumerate}
\end{lem}
\begin{proof}
We first prove (1). For the sake of contradiction, assume there is $k\in\omega$ such that $K=M'(k)\setminus N'(k)\neq0$. Since $K\le M'(k)=\llbracket k\in\dot{M}'\rrbracket$, we have $K\Vdash k\in\dot{M}'$, so by Corollary \ref{homodotgm} it holds $K\Vdash k\in\dot{M}$. Similarly, as $K\wedge N'(k)=0$ and $K\le q$, we get that $K\wedge N(k)=0$ and hence that $K\Vdash k\not\in\dot{N}$. Summing up, $K\le q$ and $K\Vdash k\in\dot{M}\setminus\dot{N}$.

Let now $G$ be an $\mathbb{M}_\kappa$-generic filter containing both $q$ and $K$. In $V[G]$, we have $k\in\dot{M}_G\setminus\dot{N}_G$ (since $K\in G$) and $\dot{M}_G\subseteq\dot{N}_G$ (since $q\in G$), which is impossible. It follows that $M'(k)\le N'(k)$ for every $k\in\omega$ and thus that $M'\le N'$.

To prove (2), first note that by Corollary \ref{homodotgm} we have $q\Vdash\dot{M}'\subseteq\dot{M}$. If $q\not\Vdash\dot{M}\subseteq\dot{M}'$, then there is $r\le q$ and $k\in\omega$ such that $r\Vdash k\in\dot{M}\setminus\dot{M}'$. Since $r\le\llbracket k\in\dot{M}\rrbracket=M(k)$ and $r\le q$, we have $r\le M'(k)$, whence $r\Vdash k\in\dot{M}'$, which is a contradiction. The case of $q\Vdash\dot{N}=\dot{N}'$ is identical.
\end{proof}

\subsection{Names for reals.} It is well known that there is a correspondence between $\mathbb{M}_\kappa$-names for reals and the ground model Borel functions $f\colon 2^\kappa \to \mathbb{R}$ (see e.g. the classical article of Scott \cite{Scott},
or \cite{Abramson}). In this section we will exploit this approach, although we will rather encode names for reals using measures than functions.

Consider the $\mathbb{M}_\kappa$-name \[ \dot{r} = \big\{\langle \langle \alpha,i \rangle, C^i_\alpha\rangle\colon \alpha < \kappa, i\in \{0,1\}\big\}, \]  where  $C^i_\alpha = \{x\in 2^\kappa\colon x(\alpha)=i\}$. We say that $\dot{r}$ is the name for \emph{the generic element
of $2^\kappa$} (or \emph{the generic random real} in case $\kappa = \omega$). Notice that 
$\llbracket \dot{r} \in C^i_\alpha \rrbracket = C^i_\alpha$ for every $\alpha<\kappa$ and $i\in \{0,1\}$. In a generic extension $V[G]$, $\dot{r}_G$ is the unique element of
$2^\kappa$ contained in every
element of the generic filter. Notice that the generic filter on $\mathbb{M}_\kappa$ induces the unique generic element of $2^\kappa$ and \textit{vice versa} (see e.g. \cite[page 6]{Sakae}).


We will need the following standard lemma. (As usual, \textit{a.e.} is the abbreviation for \textit{almost everywhere}.)

\begin{lem}\label{f-name}
	Let $f\colon 2^\kappa \to \mathbb{R}$ be a Borel function in the ground model $V$, $A\in\wp(\mathbb{R})\cap V$\footnote{Note that speaking about $f$ and $A$ we mean their ground model Borel codes, not the actual sets.}, and $y\in\mathbb{R}\cap V$. Let $p\in \mathbb{M}_\kappa$. Then
	\begin{enumerate}
		\item $(p \Vdash f(\dot{r}) \in A) \iff p \setminus \{z\in 2^\kappa\colon f(z)\in A\} \in \mathcal{N}_\kappa$;
		\item $(p\Vdash f(\dot{r})=y) \iff f=y$, $\lambda_\kappa$-a.e. on $p$.
	\end{enumerate}
\end{lem}
\begin{proof}	
	For (1), just notice that $p \Vdash f(\dot{r})\in A$ is equivalent to $p\Vdash \dot{r}\in f^{-1}[A]$ which is on the other hand equivalent to $p\leq f^{-1}[A]$. (2) follows from (1) by setting $A=\{y\}$.
\end{proof}

\begin{cor}\label{f-name-ineq}
	Let $f_1,f_2\colon 2^\kappa\to\mathbb{R}$ be Borel functions in the ground model $V$. Then, $f_1\le f_2$ $\lambda_\kappa$-a.e if and only if $\Vdash_{\mathbb{M}_\kappa}f_1(\dot{r})\le f_2(\dot{r})$.
\end{cor}
\begin{proof}
Let $\dot{x}$ be an $\mathbb{M}_\kappa$-name for an element of $\mathbb{R}$	such that $\Vdash_{\mathbb{M}_\kappa}\dot{x} = f_2(\dot{r})-f_1(\dot{r})$. Let $A=(-\infty,0)$. 

	Assume that $f_1\le f_2$ $\lambda_\kappa$-a.e. If there is $p\in\mathbb{M}_\kappa$ such that $p\Vdash\dot{x}\in A$, then by Lemma \ref{f-name}.(1) applied for $f=f_2-f_1$ we have $\lambda_\kappa(p)=0$ (since $\lambda_\kappa(f^{-1}[A])=0$), which is impossible, so $1\Vdash f_1(\dot{r})\le f_2(\dot{r})$.

	Conversely, if $\Vdash_{\mathbb{M}_\kappa}f_1(\dot{r})\le f_2(\dot{r})$, then for every $p\in\mathbb{M}_\kappa$ we have $p\Vdash\dot{x}\not\in A$, hence by Lemma \ref{f-name}.(1) we get that $p \setminus \{z\in 2^\kappa\colon (f_2-f_1)(z)\in A\} \not\in \mathcal{N}_\kappa$. It follows that for the condition $p=\{z\in 2^\kappa\colon (f_2-f_1)(z)\in A\}$ we have $p\in\mathcal{N}_\kappa$ and hence $f_1\le f_2$ $\lambda_\kappa$-a.e.
\end{proof}

To each ground model $\sigma$-additive Borel measure on $2^\kappa$ which is absolutely continuous with respect to $\lambda_\kappa$ one can linearly assign a real in any $\mathbb{M}_\kappa$-generic extension.

\begin{df}\label{names-as-measures}
Let $\mu$ be a Borel measure on $2^\kappa$ such that $\mu
\ll \lambda_\kappa$. Then, by
$\dot{r}(\mu)$ we will mean an $\mathbb{M}_\kappa$-name for a real prepared in the following way. By Radon--Nikodym theorem there is a Borel function $f\colon 2^\kappa \to \mathbb{R}$ such that $\mu(A) = \int_A f \ d\lambda_\kappa$ for every $A\in
Bor(2^\kappa)$. Now, let $\dot{r}(\mu)$ be an $\mathbb{M}_\kappa$-name such that $\Vdash_{\mathbb{M}_\kappa}\dot{r}(\mu)=f(\dot{r})$, where $\dot{r}$ is the generic random real. 
\end{df}

\begin{rem}\label{dodawanie-nazw} 
	Notice that $\dot{r}(\mu)$ is uniquely determined by $\mu$ (despite the fact that the Radon--Nikodym derivative $f$ is determined modulo $\mathcal{N}_\kappa$). Also, if $\mu$, $\nu_0$, $\nu_1$ are $\sigma$-additive Borel measures which are absolutely continuous with respect to
	$\lambda_\kappa$, and $\alpha_0$, $\alpha_1$ are non-negative real numbers, $\mu = \alpha_0\nu_0 + \alpha_1\nu_1$, then by the linearity of the Radon--Nikodym derivative it holds $\Vdash_{\mathbb{M}_\kappa} \dot{r}(\mu) =
	\alpha_0\dot{r}(\nu_0) + \alpha_1\dot{r}(\nu_1)$. 
\end{rem}

\subsection{Names for measures induced by ultrafilters.\label{name-for-measure}}
All ultrafilters on $\omega$ considered in this and following sections are assumed to be non-principal.

Fix an ultrafilter $\mathcal{U}$ on $\omega$ (in the ground model).
Following the ideas of Solovay \cite{Solovay}, we are going to define an $\mathbb{M}_\kappa$-name for a finitely additive measure on $\omega$, which is in a sense generated by $\mathcal{U}$.

Let $\dot{M}$ be an $\mathbb{M}_\kappa$-name for a subset of $\omega$ (induced by an element $M\in\mathbb{M}_\kappa^\omega$ as described in Section \ref{sec:names-subsets}) and define the function $\mu_{M}\colon\mathbb{M}_\kappa\to[0,1]$ in the following way:
\[ \mu_{M}(B) = \lim_{k\to \mathcal{U}} \lambda_\kappa (\llbracket k\in \dot{M} \rrbracket \wedge B). \]
It is immediate that $\mu_{M}$ is a finitely additive measure on $\mathbb{M}_\kappa$. Obviously, $\mu_{\check{\emptyset}}=0$ and $\mu_{\check{\omega}}=\lambda_\kappa$.

\begin{lem} For each $\mathbb{M}_\kappa$-name $\dot{M}$ for a subset of $\omega$, the measure $\mu_{M}$ is $\sigma$-additive and $\mu_{M}\ll \lambda_\kappa$.
\end{lem}

\begin{proof} Clearly, $\mu_{M} \leq \lambda_\kappa$. It implies, of course, that $\mu_{M}\ll \lambda_\kappa$ but also that $\mu_{M}$ is continuous from above at $\emptyset$. Indeed, let $(B_n)$ be a decreasing sequence of
	Borel subsets of $2^\kappa$ such that $\bigcap_{n\in\omega} B_n=\emptyset$. Then, 
	$\lambda_\kappa(B_n) \geq \mu_{M}(B_n)$
	for every $n\in \omega$ and so, by the continuity of $\lambda_\kappa$, $\lim_{n\to\infty} \mu_{M}(B_n)  = 0$. It follows that $\mu_{M}$ is $\sigma$-additive, too.
\end{proof}

Now, with $\mu_{M}$ we can associate an $\mathbb{M}_\kappa$-name for a real $\dot{r}(\mu_{M})$ in the way described in Definition \ref{names-as-measures}. Finally, let $\dot{\mu}_\mathcal{U}$ be such an $\mathbb{M}_\kappa$-name for a function
$\wp(\omega)\to[0,1]$ such that for every $\mathbb{M}_\kappa$-name $\dot{M}$ for a subset of $\omega$ we have
\[\Vdash_{\mathbb{M}_\kappa} \dot{\mu}_\mathcal{U}(\dot{M}) = \dot{r}(\mu_{M}).\]

The following fact is proved in  \cite{Solovay} (see also \cite{Kunen-special-points}).

\begin{prop} \[ \Vdash_{\mathbb{M}_\kappa}``\dot{\mu}_\mathcal{U} \mbox{ is a finitely additive probability measure on }\omega". \]
\end{prop}
\begin{proof}
Let $G$ be an $\mathbb{M}_\kappa$-generic filter over the ground model $V$. We need to show that in the extension $V[G]$ the interpretation $\rho=(\dot{\mu}_\mathcal{U})_G$ of the name $\dot{\mu}_\mathcal{U}$ through $G$ is a finitely additive probability measure on $\omega$. It is immediate that $\rho(\emptyset)=0$ and $\rho(\omega)=1$.

Let $A,B\in V[G]$ be two subsets of $\omega$ such that $A\cap B=\emptyset$. By Lemma \ref{names}, there are in $V$ two elements $M,N\in\mathbb{M}_\kappa^\omega$ such that $\dot{M}_G=A$ and $\dot{N}_G=B$. By taking $N\setminus M$ instead of $N$ and appealing to Corollary \ref{homodotgm}, we may assume that $M\wedge N=0$
, hence $\lambda_\kappa(M(k) \vee N(k)) = \lambda_\kappa(M(k)) + \lambda_\kappa(N(k))$ for each $k\in\omega$. By the additivity of the limit $\lim_{n\to \mathcal{U}}$, we have $\mu_{M \lor N} = \mu_{M} + \mu_{N}$. By Remark \ref{dodawanie-nazw} we are done.
\end{proof}

We call $\dot{\mu}_\mathcal{U}$ \emph{the $\mathbb{M}_\kappa$-name for the Solovay measure associated to the ultrafilter $\mathcal{U}$}. If $G$ is an $\mathbb{M}_\kappa$-generic filter over $V$, then the evaluation $(\dot{\mu}_\mathcal{U})_G$ in $V[G]$ is called
\emph{the Solovay measure (associated to $\mathcal{U}$)}. Usually, we will simply write $\dot{\mu}_\mathcal{U}$ instead of $(\dot{\mu}_\mathcal{U})_G$, in particular, when we do not have in mind any specific extension $V[G]$.

The next proposition shows that, in $V^{\mathbb{M}_\kappa}$, $\dot{\mu}_\mathcal{U}$ extends the ground model one-point measure $\delta_{\mathcal{U}}$ on $\omega$, that is, the measure defined for every $A\in\wp(\omega)\cap V$ as follows: $\delta_{\mathcal{U}}(A)=1$ if $A\in\mathcal{U}$, and $\delta_{\mathcal{U}}(A)=0$ otherwise.

\begin{prop}\label{rho_extends}
For every $A\in\wp(\omega)\cap V$ we have $\Vdash_{\mathbb{M}_\kappa}\dot{\mu}_\mathcal{U}(A)=\delta_{\mathcal{U}}(A)$.
\end{prop}
\begin{proof}
For every $A\in\mathcal{U}$, $\mu_{\check{A}}=\lambda_\kappa$ and $\mu_{(\omega\setminus A)\check{\ }}=0$, so $\Vdash_{\mathbb{M}_\kappa}\dot{\mu}_\mathcal{U}(A)=1$ and $\Vdash_{\mathbb{M}_\kappa}\dot{\mu}_\mathcal{U}(\omega\setminus A)=0$.
\end{proof}

Consequently, we obtain that if, in $V$, $\mathcal{A}\subseteq\wp(\omega)$, $\varphi\colon\mathcal{A}\to[0,1]$ is a function, and there is $A\in\mathcal{A}\cap\mathcal{U}$ such that $0<\varphi(A)<1$, then, in $V^{\mathbb{M}_\kappa}$, $\dot{\mu}_\mathcal{U}$ does not extend $\varphi$. In particular, $\dot{\mu}_\mathcal{U}$ is not an extension of any non-atomic ground model measure on $\omega$.

\subsection{P-measure in the random model.}

For an ultrafilter $\mathcal{U}$ on $\omega$ by $\dot{\mu}_\mathcal{U}$ we denote the $\mathbb{M}_\kappa$-name for the Solovay measure on $\omega$ associated to $\mathcal{U}$. Our main result is the following.

\begin{thm}\label{thm:random_pp_ap} If $\mathcal{U}$ is a P-point, then $\Vdash_{\mathbb{M}_\kappa}$``$\dot{\mu}_\mathcal{U}$ is a P-measure''. 
\end{thm}

Since there are P-points under the assumption of the Continuum Hypothesis, the following result follows.

\begin{cor}\label{cor:random_ap}
	In the random model, obtained by forcing with $\mathbb{M}_{\omega_2}$ over a model satisfying $\mathsf{CH}$, there is a P-measure.%
\end{cor}
By the discussion after Proposition \ref{rho_extends} we get also that $\Vdash_{\mathbb{M}_\kappa}$``$\dot{\mu}_\mathcal{U}$ does not extend the asymptotic density''. 

Recall that the support of a P-measure on $\omega$ is a nowhere dense ccc closed P-set (see Lemma \ref{lem:basic}).

\begin{cor}\label{cor:random_ap2}
In the random model, there is a nowhere dense ccc closed P-set in $\omega^*$.
\end{cor}

To prove Theorem \ref{thm:random_pp_ap} we will need a lemma.

\begin{lem}\label{AP1} Suppose $\mathcal{U}$ is a P-point. Let $\nu$ be the measure on $\mathbb{M}^\omega_\kappa$ defined for every $M\in\mathbb{M}_\kappa^\omega$ by the formula $\nu(M) = \lim_{i \to \mathcal{U}} \lambda_\kappa(M(i))$. Then, $\nu$  satisfies the following
	property: whenever $(M_n)$ is a sequence from $\mathbb{M}^\omega_\kappa$ such that $M_{n+1} \le M_n$ for each $n\in\omega$, then there is $M\in \mathbb{M}^\omega_\kappa$ such that
	\begin{enumerate}
		\item $M\le^*M_n$ for each $n\in\omega$, and
		\item $\nu(M) = \lim_{n\to\infty} \nu(M_n)$.
	\end{enumerate}
\end{lem}

\begin{proof}
	Let $(M_n)$ be a decreasing sequence in $\mathbb{M}^\omega_\kappa$.
	For each $n\in\omega$ let $r_n = \nu(M_n)$ and put  
	\[ U_n = \big\{k\in\omega\colon\ |\lambda_\kappa(M_n(k)) - r_n| < 1/n \big\}.\]
	Notice that $U_n \in \mathcal{U}$ for every $n\in\omega$.

	Since $\mathcal{U}$ is a P-point, there is $U\in \mathcal{U}$ such that $U\subseteq^* U_n$ for every $n\in\omega$. It follows that for every $n\in\omega$ there is $l_n\in U$ such that $U\setminus\{0,\ldots,l_n-1\}\in U_n$. In particular, $l_n\in U_n$. Increasing sufficiently each $l_n$, we may assume that $l_{n+1}>l_n\ge n$ for every $n\in\omega$.

For each $k<l_0$ set $M(k)=1$. For each $k\ge l_0$ there is $n_k\in\omega$ such that $l_{n_k}\le k<l_{n_k+1}$, so we may set $M(k)=M_{n_k}(k)$.

We will check that the sequence $M=(M(k))$ is as desired. 

Since the sequence $(M_n)$ is decreasing, for each $n\in\omega$ and $k\ge l_n$ we have $M(k)\le M_n(k)$, so (1) holds. We claim that (2) holds, too, that is, that $\nu(M)=\lim_{n\to\infty}\nu(M_n)$. Let $r=\nu(M)$. By (1), for every $n\in\omega$ we have:
\[\nu(M_n)=\lim_{k\to\mathcal{U}}\lambda_\kappa(M_n(k))\ge\lim_{k\to\mathcal{U}}\lambda_\kappa(M(k))=\nu(M),\]
which yields that $r_n\ge r$. It follows that $\lim_{n\to\infty}r_n\ge r$ (note that the sequence $(r_n)$ is decreasing). Assume that there exists $\varepsilon>0$ such that $r_n>r+\varepsilon$ for every $n\in\omega$. Put
\[A=\big\{k\in\omega\colon\ |\lambda_\kappa(M(k))-r|<\varepsilon/2\big\}.\]
Of course, $A\in\mathcal{U}$, so $B=A\cap U\in\mathcal{U}$. Let $k\in B$ be sufficiently large, so that there exists $n_k\in\omega$ such that $l_{n_k}\le k<l_{n_k+1}$ and $1/n_k<\varepsilon/2$. It follows that $k\in U_{n_k}$, so 
\[|\lambda_\kappa(M_{n_k}(k))-r_{n_k}|<1/n_k<\varepsilon/2,\]
hence
\[\tag{$*$}\lambda_\kappa(M_{n_k}(k))>r_{n_k}-\varepsilon/2>r+\varepsilon/2.\]
On the other hand, since $k\in A$, we also have that
\[|\lambda_\kappa(M_{n_k}(k))-r|<\varepsilon/2,\]
so
\[\tag{$**$}\lambda_\kappa(M_{n_k}(k))<r+\varepsilon/2.\]
But then, by ($*$) and ($**$), we get that
\[\lambda_\kappa(M_{n_k}(k))>r+\varepsilon/2>\lambda_\kappa(M_{n_k}(k)),\]
a contradiction. Thus, $\lim_{n\to\infty}r_n=r$, and the proof is finished.
\end{proof}

We are in the position to prove the theorem.

\begin{proof}[Proof of Theorem \ref{thm:random_pp_ap}]
	We work in the ground model $V$. Suppose for the sake of contradiction that there is a condition $p\in\mathbb{M}_\kappa$ forcing that $\dot{\mu}_\mathcal{U}$ is not a P-measure. Then, by Lemma \ref{epsilon}, the Maximum Principle, and Lemma \ref{names}, there is a condition $q\le p$, a rational $\varepsilon>0$, and a sequence $(M_n)$ of elements of $\mathbb{M}_\kappa^\omega$ such that
	\[ q \Vdash ``(\dot{M}_n)\mbox{ is a decreasing sequence of subsets of }\omega",\]
and
	\[ q \Vdash``\mbox{if }\dot{M}\mbox{ is a pseudointersection of }(\dot{M}_n)\mbox{, then } \forall n\in\omega \ \dot{\mu}_\mathcal{U}(\dot{M}) < \dot{\mu}_\mathcal{U}(\dot{M}_n)-\varepsilon". \]
	For each $n\in\omega$ define $N_n\in \mathbb{M}^\omega_\kappa$ by $N_n=M_n\wedge\vec{q}$. Then, by Lemma \ref{restr_incl}.(1), the sequence $(N_n)$ is decreasing and we can apply Lemma \ref{AP1} to find $M\in \mathbb{M}^\omega_\kappa$ such that
	\begin{enumerate}
		\item $M\le^* N_n$ for each $n\in\omega$, and
		\item $\nu(M) = \lim_{n\to\infty}\nu(N_n)$, where the measure $\nu$ is defined as in Lemma \ref{AP1}.
	\end{enumerate}

	Since, by Lemmas \ref{almost} and \ref{restr_incl}.(2), $q \Vdash ``\dot{M}\mbox{ is a pseudointersection of }(\dot{M}_n)"$, for every $n\in\omega$ we have:
	\[ q \Vdash \dot{\mu}_\mathcal{U}(\dot{M}) < \dot{\mu}_\mathcal{U}(\dot{M}_n) - \varepsilon. \]
	Let $f,f_n\colon 2^\kappa \to \mathbb{R}$ be Borel functions associated to $\mu_{M}$ and $\mu_{M_n}$, $n\in\omega$, via the Radon--Nikodym theorem (see Definition \ref{names-as-measures}). 
 	Then, for every $n\in\omega$ we have	
	\[\tag{$*$} f < f_n - \varepsilon \quad\mbox{almost everywhere on }q, \]
	which follows from Lemma \ref{f-name}.(1) (applied to the function $f-f_n$ and the set $A=(-\infty,-\eps)$).

Integrating both sides of inequality ($*$) (over $q$), we get for every $n\in\omega$ that
	\[ \int_q f \ d\lambda_\kappa < \int_q f_n \ d\lambda_\kappa - \varepsilon \lambda_\kappa(q).\]
	Then, by the definitions of $\mu_{M}$ and $\mu_{M_n}$, we have:
   \[ \mu_{M}(q) < \mu_{M_n}(q) - \varepsilon \lambda_\kappa(q), \]
	and so
	 \[ \nu(M \wedge \vec{q}) < \nu(M_n \wedge \vec{q}) - \varepsilon \lambda_\kappa(q). \]
	Since for each $n\in\omega$ we have $N_n = M_n \wedge \vec{q}$ and $M \le^* N_n\le\vec{q}$, it holds
	\[ \nu(M) < \nu(N_n) - \varepsilon\lambda_\kappa(q),\]
	for every $n\in\omega$, which contradicts (2) above.
\end{proof}

The converse to Theorem \ref{thm:random_pp_ap} also holds.

\begin{thm}\label{pmeasure_extends_ppoint} Let $\mathcal{U}$ be a ground model ultrafilter on $\omega$ such that $\Vdash_{\mathbb{M}_\kappa}$``$\dot{\mu}_{\mathcal{U}}$ is a P-measure''. Then, $\mathcal{U}$ is a P-point in $V$.
\end{thm}

\begin{proof} Suppose for the sake of contradiction that $\mathcal{U}$ is not a P-point in $V$ and let $(A_n)$ be a sequence witnessing it, i.e. $(A_n)$ is a decreasing sequence such that:
	\begin{itemize}
		\item[(a)] $A_0=\omega$ and $A_n \in \mathcal{U}$ for each $n\in\omega$, 
		\item[(b)] for every $A\in\wp(\omega)\cap V$, if $A \setminus A_n$ is finite for each $n\in\omega$, then $A\notin \mathcal{U}$. 
	\end{itemize}
	Since $\Vdash_{\mathbb{M}_\kappa}$``$\dot{\mu}_\mathcal{U}$ is a P-measure'', by (a) and Proposition \ref{rho_extends}, there is $M\in\mathbb{M}_\kappa^\omega$ such that $\Vdash_{\mathbb{M}_\kappa}\dot{\mu}_\mathcal{U}(\dot{M})=1$ and $\Vdash_{\mathbb{M}_\kappa}$``$\dot{M} \setminus A_n$ is finite for each $n\in\omega$''. In particular, for every $n\in\omega$ we have that $\Vdash_{\mathbb{M}_\kappa}$``$\dot{M}\cap(A_n\setminus A_{n+1})$ is finite''. Then, by Lemma \ref{almost},
\[ \lambda_\kappa\Big( \bigwedge_{l\in\omega} \bigvee_{\substack{k\in A_n\setminus A_{n+1}\\k>l}} M(k)\Big) = 0 \]
for each $n\in\omega$.
So, for each $n\in\omega$, there is a finite subset $F_n \subseteq A_n\setminus A_{n+1}$ such that 
\[ \lambda_\kappa\Big(\bigvee_{k\in (A_n\setminus A_{n+1})\setminus F_n} M(k)\Big) < 1/2^{n+2}. \]
Let $A = \bigcup_{n\in\omega} F_n$. Obviously, $A\in V$ and $A\setminus A_n$ is finite for every $n\in\omega$ (since $A\setminus A_n\subseteq\bigcup_{i=0}^{n-1} F_i$), so by (b) we have $A\not\in\mathcal{U}$ and hence $A^c\in\mathcal{U}$. By Proposition \ref{rho_extends}, $\Vdash_{\mathbb{M}_\kappa}\dot{\mu}_\mathcal{U}(A^c)=1$. Let  
\[ p = \bigvee_{k \in A^c} M(k) = \bigvee_{n\in \omega} \bigvee_{k\in(A_n\setminus A_{n+1})\setminus F_n} M(k). \] 
Notice that $\lambda(p)\leq 1/2$ and so, in particular, $p\ne 1$. Hence, setting $q=1\setminus p$, we have $q\neq0$. Then $q\Vdash \dot{M} \cap A^c = \emptyset$, and so $q\Vdash \dot{M} \subseteq A$. Since $\Vdash_{\mathbb{M}_\kappa}\dot{\mu}_\mathcal{U}(\dot{M})=1$, we get that  $q\Vdash\dot{\mu}_\mathcal{U}(A)=1$. But $\Vdash_{\mathbb{M}_\kappa}\dot{\mu}_\mathcal{U}(A^c)=1$, so $q\Vdash\dot{\mu}_\mathcal{U}(A^c)=1$, too, which is a contradiction. 
\end{proof}

\begin{cor}
Let $\mathcal{U}$ be a ground model ultrafilter on $\omega$. Then, $\mathcal{U}$ is a P-point in $V$ if and only if $\Vdash_{\mathbb{M}_\kappa}$``$\dot{\mu}_\mathcal{U}$ is a P-measure''.
\end{cor}

\section{Non-atomicity of Solovay measures}

Recall that each P-point induces a P-measure which is a one-point measure. So, if we could show that the P-measures whose existence we have shown in the previous section can be one-point measures, then we would prove the existence of a P-point
in the random model. Unfortunately, as we will see, all Solovay measures are non-atomic.

Suppose $\mathcal{U}$ is an ultrafilter on $\omega$, in $V$. Let $G$ be an $\mathbb{M}_\kappa$-generic filter over $V$. Then, in $V[G]$, $\mathcal{U}$ generates a filter on $\omega$, and so a closed subset of
$\beta\omega$. Denote this set by $K_\mathcal{U}$. In other words, in $V[G]$,
\[ K_\mathcal{U} = \bigcap_{X\in \mathcal{U}}[X].\]
(Recall, that here $[X]$ is the clopen subset of $\beta\omega$ associated to a subset $X$ of $\omega$, i.e. $[X]=[X]_{\wp(\omega)}$.) Now, still in $V[G]$, the finitely additive Solovay measure $(\dot{\mu}_\mathcal{U})_G$ associated with $\mathcal{U}$ can be uniquely extended to a regular Borel $\sigma$-additive probability measure $\nu_\mathcal{U}$ on $\beta\omega$. Let $\dot{K}_\mathcal{U}$ and $\dot{\nu}_\mathcal{U}$ denote such $\mathbb{M}_\kappa$-names in $V$ that in every $\mathbb{M}_\kappa$-generic extension $V[G]$ we have $(\dot{K}_\mathcal{U})_G=K_\mathcal{U}$ and $(\dot{\nu}_\mathcal{U})_G=\nu_\mathcal{U}$. It appears that each measure of the form $\nu_\mathcal{U}$ is carried by $K_\mathcal{U}$, i.e., that $\supp(\nu_\mathcal{U})\subseteq K_\mathcal{U}$. 

\begin{prop}\label{nuu_ku} For each ultrafilter $\mathcal{U}$ on $\omega$, \[ \Vdash_{\mathbb{M}_\kappa} \dot{\nu}_\mathcal{U}(\dot{K}_\mathcal{U}) = 1.\]
\end{prop}

\begin{proof}
By Proposition \ref{rho_extends} for every $X\in\mathcal{U}$ we have $\Vdash_{\mathbb{M}_\kappa}\dot{\mu}_\mathcal{U}(\check{X})=1$, so $\Vdash_{\mathbb{M}_\kappa}\dot{\nu}_\mathcal{U}([\check{X}])=1$. The measure $\dot{\nu}_\mathcal{U}$ is a Radon measure in $V^{\mathbb{M}_\kappa}$, hence it is $\tau$-additive (by \cite[Proposition 7.2.2.(i)]{BogMT2}). Since $\{[X]\colon X\in\mathcal{U}\}$ is a decreasing net of closed sets (indexed by the directed set $(\mathcal{U},\supseteq)$), $\Vdash_{\mathbb{M}_\kappa}\dot{\nu}_\mathcal{U}(\dot{K}_\mathcal{U})=1$ (by \cite[Proposition 7.2.5]{BogMT2}).
\end{proof}

We will now show that Solovay measures are always non-atomic, regardless of an associated ground model ultrafilter. This fact will be an immediate consequence of the existence of a family $\mathcal{M}=\{M_s\in\mathbb{M}_\kappa^\omega\colon s\in 2^{<\omega}\}$ having for every $s\in 2^{<\omega}$ the following properties:
\begin{enumerate}
	\item $M_s$ is a disjoint union of $M_{s^\smallfrown0}$ and $M_{s^\smallfrown1}$, and
	\item if $\mathcal{U}$ is a ground model ultrafilter, then $\Vdash_{\mathbb{M}_\kappa}\dot{\mu}_\mathcal{U}(\dot{M}_s)=1/2^{|s|}$.
\end{enumerate}
The definition of $\mathcal{M}$ is as follows. We let $M_\epsilon=\vec{1}$ and for each $s\in 2^{<\omega}$, $|s|>0$, and $k\in\omega$, we put:
\[M_s(k)=\big\{x\in 2^\kappa\colon\ x(k+i)=s(i)\text{ for all }0\le i<|s|\big\}.\]
Note that trivially $\lambda_\kappa(M_s(k))=1/2^{|s|}$. 

Property (1) of $\mathcal{M}$ is obviously satisfied. We will now show that (2) also holds.

\begin{lem}\label{Ms_measure}
Let $\mathcal{U}$ be a ground model ultrafilter. For every $n\in\omega$ and $s\in 2^n$, the Radon--Nikodym derivative $f_{M_s}$ of $\mu_{M_s}$ is $\lambda_\kappa$-a.e. equal to $1/2^n$, and so it holds
\[\Vdash_{\mathbb{M}_\kappa}\dot{\mu}_\mathcal{U}(\dot{M}_s)=1/2^n.\]
\end{lem}
\begin{proof}
Let $C\subseteq 2^\kappa$ be a clopen. Then it depends on finitely many coordinates, i.e. there is a finite set $I\subseteq \kappa$ such that for every $x\in C$ and
$y\in 2^\kappa$, if $x(\alpha)=y(\alpha)$ for all $\alpha\in I$
then $y\in C$.

It follows that there is $K\in\omega$ such that for every $k\in\omega$, $k>K$, we have $k\not\in I$. Hence, for every such $k$ it holds $\lambda_\kappa(M_s(k)  \wedge C) = 1/2^n \lambda_\kappa(C)$, and so 
	\[ \lim_{k\to \infty} \lambda_\kappa( M_s(k) \wedge C) = 1/2^n \lambda_\kappa(C).\]

Now, let $B\subseteq 2^\kappa$ be a Borel set. Fix $\varepsilon>0$. There is a clopen $C\subseteq 2^\kappa$ such that $\lambda_\kappa( B\triangle C)<\varepsilon$. Let $K\in\omega$ be as in the previous paragraph, so that $\lambda_\kappa(M_s(k)  \wedge C) = 1/2^n \lambda_\kappa(C)$ for every $k\in\omega$, $k>K$. Then, for every such $k$ we have:
	\[ |\lambda_\kappa( M_s(k) \wedge B) - 1/2^n \lambda_\kappa(B)|\le\]
	\[\le|\lambda_\kappa( M_s(k) \wedge B) - \lambda_\kappa( M_s(k) \wedge C)|+ |\lambda_\kappa( M_s(k) \wedge C) - 1/2^n \lambda_\kappa(B)|\le\]
	\[\le\lambda_\kappa( M_s(k) \wedge (B\triangle C))+|1/2^n\lambda_\kappa(C) - 1/2^n \lambda_\kappa(B)|<\]
	\[<(1 + 1/2^n) \varepsilon.\]
	Since $\varepsilon$ was arbitrary, we get that
	\[ \lim_{k\to \infty} \lambda_\kappa( M_s(k) \wedge B) = 1/2^n \lambda_\kappa(B).\]

	It means that $\mu_{M_s}=1/2^nd\lambda_\kappa$, so the Radon--Nikodym derivative $f_{M_s}$ of $\mu_{M_s}$ is almost everywhere constantly equal to $1/2^n$. Hence, $\Vdash_{\mathbb{M}_\kappa} \dot{\mu}_\mathcal{U}(\dot{M}_s) = 1/2^n$.
\end{proof}

\begin{cor}\label{non-atomic}
For every ultrafilter $\mathcal{U}$ on $\omega$, 
\[\Vdash_{\mathbb{M}_\kappa}``\dot{\mu}_\mathcal{U}\text{ is non-atomic}".\]
\end{cor}
\begin{proof}
Fix $\varepsilon>0$ and let $n\in\omega$ be such that $\varepsilon>1/2^n$. Consider $\{\dot{M}_s\colon\ s\in 2^n\}$ to obtain a desired partition of $\omega$ in $V^{\mathbb{M}_\kappa}$.
\end{proof}

In Lemma \ref{Ms_measure} we showed that, for every $s\in 2^{<\omega}$, the function $f_{M_s}$ is constantly equal to $1/2^{|s|}$ ($\lambda_\kappa$-a.e.). It appears that, exploiting this fact, one can show that for every ground model $\alpha\in[0,1]$ there exists an element $M_\alpha\in\mathbb{M}_\kappa^\omega$ such that the function $f_{M_\alpha}$ is constantly equal to $\alpha$ ($\lambda_\kappa$-a.e) and hence, for every ground model ultrafilter $\mathcal{U}$, $\Vdash_{\mathbb{M}_\kappa}\dot{\mu}_\mathcal{U}(\dot{M}_\alpha)=\alpha$.

We construct elements $M_\alpha$'s as follows. Set $M_0=\check{\emptyset}$ and notice that for every $\alpha\in[0,1]$ there is a (finite or infinite) antichain $I_\alpha\in2^{<\omega}$ such that $\alpha=\sum_{s\in I_\alpha}1/2^{|s|}$ and set for every $k\in\omega$:
\[M_\alpha(k)=\bigvee_{s\in I_\alpha}M_s(k).\]
Note that for each distinct $s,t\in I_\alpha$ we have $M_s(k)\cap M_t(k)=\emptyset$. Additionally, put
\[\dot{\mathcal{M}}_\sigma=\{M_\alpha\colon\ \alpha\in[0,1]\}.\]

\begin{thm}\label{gm_reals}
Let $\mathcal{U}$ be a ground model ultrafilter. For every $\alpha\in[0,1]\cap V$, the Radon--Nikodym derivative $f_{M_\alpha}$ of $\mu_{M_\alpha}$ is $\lambda_\kappa$-a.e. equal to $\alpha$, and so we have:
\[\Vdash_{\mathbb{M}_\kappa}\dot{\mu}_\mathcal{U}(\dot{M}_\alpha)=\alpha.\]
In particular, $\Vdash_{\mathbb{M}_\kappa}\dot{\mu}_\mathcal{U}[\dot{\mathcal{M}}_\sigma]=[0,1]\cap V$.
\end{thm}
\begin{proof}
Fix $\alpha\in[0,1]\cap V$ and a Borel $B\subseteq 2^\kappa$. We will show that
\[\lim_{k\to\infty}\lambda_\kappa(M_\alpha(k)\wedge B)=\alpha\lambda_\kappa(B).\]
For $\alpha=0$ and $M_0$ we are obviously done, so assume that $\alpha>0$. If $I_\alpha$ is finite, then the equality easily follows from the additivity of $\lambda_\kappa$ and $\lim_{k\to\infty}$. So assume that $I_\alpha$ is infinite. Enumerate $I_\alpha=(s_n)$ and for each $n\in\omega$ set $\alpha_n=1/2^{|s_n|}$. Then, $\alpha=\sum_{n\in\omega}\alpha_n$. For each $n,k\in\omega$ let
\[f_k(n)=\lambda_\kappa(M_{s_n}(k)\wedge B).\]
Hence, for each $k\in\omega$ we have $f_k\colon\omega\to[0,1]$ and $\sum_{n\in\omega}f_k(n)\le 1$. For each $n\in\omega$, by (the proof of) Lemma \ref{Ms_measure}, we also have: \[\lim_{k\to\infty}f_k(n)=\alpha_n\lambda_\kappa(B).\]
Finally, for each $k,n\in\omega$ it holds $|f_k(n)|\le1/2^{|s_n|}=\alpha_n$, so for the function $g\colon\omega\to[0,1]$, given for every $n\in\omega$ by the formula $g(n)=\alpha_n$, we have $|f_k(n)|\le g(n)$ for every $k,n\in\omega$ and $\sum_{n\in\omega}g(n)\le1$.

Using the $\sigma$-additivity of $\lambda_\kappa$ and Lebesgue's Dominated Convergence Theorem (applied to the functions $f_k$'s and $g$ and the counting measure on $\omega$), we get:
\[\lim_{k\to\infty}\lambda_\kappa(M_\alpha(k)\wedge B)=\lim_{k\to\infty}\sum_{n\in\omega}\lambda_\kappa(M_{s_n}(k)\wedge B)=\sum_{n\in\omega}\lim_{k\to\infty}\lambda_\kappa(M_{s_n}(k)\wedge B)=\]
\[=\sum_{n\in\omega}\alpha_n\lambda_\kappa(B)=\lambda_\kappa(B)\sum_{n\in\omega}\alpha_n=\alpha\lambda_\kappa(B),\]
which is what we wanted.

To finish the proof, note that as previously the Radon--Nikodym derivative $f_{M_\alpha}$ is constantly equal to $\alpha$ ($\lambda_\kappa$-a.e.).
\end{proof}

Theorem \ref{gm_reals} asserts that the chunk $\dot{\mu}_\mathcal{U}\restriction\dot{\mathcal{M}}_\sigma$ of every Solovay measure $\dot{\mu}_\mathcal{U}$ attains every ground model real value and is described by simple $\mathbb{M}_\kappa$-names (the elements of $\dot{\mathcal{M}}_\sigma$). It also shows that every two Solovay measures, whatever their associated ground model ultrafilters are, have the same values on $\dot{\mathcal{M}}_\sigma$. Note that by Proposition \ref{rho_extends} the interpretations of names from $\dot{\mathcal{M}}_\sigma$ through an $\mathbb{M}_\kappa$-generic filter do not belong to the ground model (except for $M_0$ and $M_1$).

Let us finish this section with the following remark. In Corollary \ref{nonatomic_ap}, referring to the proof of the classical Sierpi\'nski theorem, we claimed that every non-atomic P-measure on $\omega$ attains every value in the interval $[0,1]$. It appears however that in the case of Solovay P-measures one can obtain the same result without actually using Sierpi\'nski's theorem---it is enough to repeat the (slightly adjusted) argument presented in the proof of Theorem \ref{gm_reals}.

\subsection{Semi-selective ultrafilters\label{semi-sele}}

Now, we will show that for a particular subclass of P-points their Solovay measures are as far from being one-point measures as possible, in the sense that their Radon extensions are strictly positive on the sets
$K_\mathcal{U}$'s (and so their supports are as big as possible in the light of Proposition \ref{nuu_ku}).

Recall the following standard definition of a semi-selective ultrafilter.

\begin{df} An ultrafilter $\mathcal{U}$ on $\omega$ is called \emph{semi-selective} if whenever $(a_n)$ is a sequence of non-negative reals such that $\lim_{n\to\mathcal{U}} a_n =0$, then there is $X\in \mathcal{U}$ such that $\sum_{n\in X}a_n < \infty$.
\end{df}

Note that an ultrafilter $\mathcal{U}$ is semi-selective if and only if it is a rapid P-point (see e.g. \cite{Canjar}). In particular, every selective ultrafilter is semi-selective. The next theorem provides two characterizations of semi-selective ultrafilters: one in terms resembling the classical Borel--Cantelli lemma, and the other in terms of the supports of Solovay measures (cf. also \cite[Theorem 3.2]{Grebik}). 

\begin{thm}\label{semi-selective} Suppose $\mathcal{U}$ is an ultrafilter on $\omega$. Then the following are equivalent:
	\begin{enumerate}
	\item $\mathcal{U}$ is semi-selective;
	\item for every probability measure $\mu$ on $2^\kappa$ and sequence $(P_k)$ in $\mathrm{Bor}(2^\kappa)/\mathcal{N}(\mu)$, if $\lim_{k\to\mathcal{U}}\mu(P_k)=0$, then there is $X\in\mathcal{U}$ such that
	\[\bigwedge_{k\in X}\bigvee_{\substack{i\in X\\i>k}}P_i = 0;\]
	\item $\Vdash_{\mathbb{M}_\kappa}$``$\dot{\nu}_\mathcal{U}$ is strictly positive on $\dot{K}_\mathcal{U}$''.
	\end{enumerate}
\end{thm}

($\mathcal{N}(\mu)$ in (2) denotes the ideal of $\mu$-null Borel subsets of $2^\kappa$.)

\begin{proof} (1) $\implies$ (2). Let $\mu$ be a probability measure on $2^\kappa$ and $(P_k)$ a sequence in $\mathrm{Bor}(2^\kappa)/\mathcal{N}(\mu)$ such that $\lim_{k\to\mathcal{U}}\mu(P_k)=0$. Since $\mathcal{U}$ is semi-selective, there is $X\in \mathcal{U}$ such that $\sum_{k\in X}\mu(P_k)< \infty$. The Borel--Cantelli lemma implies that
\[\bigwedge_{k\in X}\bigvee_{\substack{i\in X\\i>k}} P_i = 0.\]
	\medskip

	(2) $\implies$ (3). Recall that $1\Vdash\supp(\dot{\nu}_\mathcal{U})\subseteq\dot{K}_\mathcal{U}$. Fix  $p\in\mathbb{M}_\kappa$ and $M\in\mathbb{M}_\kappa^\omega$ such that $p\Vdash\dot{\nu}_\mathcal{U}([\dot{M}]) = 0$. We will show that $p\Vdash[\dot{M}]\cap\dot{K}_\mathcal{U}=\emptyset$. Set $P = M \wedge \vec{p}$.

	We claim that $\lim_{k\to \mathcal{U}}\lambda_\kappa(\llbracket k \in\dot{P}\rrbracket) =0$. Indeed, since $P\le M$, $\mu_{P}\le\mu_{M}$, so $f_{P}\le f_{M}$ $\lambda_\kappa$-a.e., and hence, by Corollary
	\ref{f-name-ineq}, $1\Vdash\dot{\mu}_{\mathcal{U}}(\dot{P})\le\dot{\mu}_{\mathcal{U}}(\dot{M})$. It follows that $p\Vdash\dot{\mu}_{\mathcal{U}}(\dot{P})=0$. Since $P\le\vec{p}$, by the definition of $f_P$ we have $f_{P}=0$ $\lambda_\kappa$-a.e. on $1\setminus p$, so by Lemma \ref{f-name}.(2)
	we have $1\setminus p\Vdash\dot{\mu}_{\mathcal{U}}(\dot{P})=0$. In total, $1\Vdash\dot{\mu}_{\mathcal{U}}(\dot{P})=0$. Finally, again by Lemma \ref{f-name}.(2), we have that $f_{P}=0$ $\lambda_\kappa$-a.e., so $\lim_{k\to \mathcal{U}}\lambda_\kappa(\llbracket k \in\dot{P}\rrbracket)=\mu_{P}(1)=0$.

 	By (2) there is $X\in\mathcal{U}$ such that
	\[\tag{$*$}\bigwedge_{k\in X}\bigvee_{\substack{i\in X\\i>k}}\llbracket i\in\dot{P}\rrbracket=0.\]
	Since $\llbracket i\in\dot{P}\rrbracket=\llbracket i\in\dot{M}\rrbracket\wedge p$ for each $i\in\omega$, ($*$) means that $p\Vdash X\cap\dot{M}=^*\emptyset$, and thus  $p\Vdash[X]\cap[\dot{M}]=\emptyset$. From $1\Vdash\dot{K}_\mathcal{U}\subseteq[X]$ (recall that $X\in\mathcal{U}$) we get that $p\Vdash\dot{K}_\mathcal{U}\cap[\dot{M}]=\emptyset$.
	\medskip

	(3) $\implies$ (1). If $\mathcal{U}$ is not semi-selective, then there is a sequence $(a_n) \in [0,1]^\omega$ such that $\lim_{n\to\mathcal{U}} a_n = 0$
	and $\sum_{n\in X} a_n = \infty$ for each $X\in \mathcal{U}$. By \cite[Lemma 1.3.23]{BJ95}, there is $M\in \mathbb{M}^\omega_\kappa$ such that $\lambda_\kappa(M(k)) = a_k$ for each $k\in \omega$ and the sequence $(M(k))$ is $\lambda_\kappa$-independent. It follows that $\lim_{k\to \mathcal{U}}\lambda_\kappa(\llbracket k \in
	\dot{M} \rrbracket) =0$ and so $1\Vdash \dot{\mu}_\mathcal{U}(\dot{M})=0$. We claim that $1\Vdash[\dot{M}]\cap K_\mathcal{U}\ne\emptyset$. Indeed, if $X\in \mathcal{U}$, then $\sum_{k\in X}a_k = \infty$ and so, by the second Borel--Cantelli lemma,
\[\bigwedge_{k\in X}
	\bigvee_{\substack{n\in X\\n>k}}M(n) = 1.\]
	But this precisely means that $1\Vdash \dot{M} \cap \check{X} \ne^* \emptyset$. As $\mathcal{U}$ is a directed family and  $1\Vdash$``$\dot{K}_\mathcal{U}$ is compact and non-empty'', $1\Vdash[\dot{M}]\cap\dot{K}_\mathcal{U} \ne \emptyset$, which is a contradiction since we have $1\Vdash \dot{\mu}_\mathcal{U}(\dot{M})=0$ and $1\Vdash \supp(\dot{\nu}_\mathcal{U})=\dot{K}_\mathcal{U}$.
\end{proof}

As a corollary we get the following well-known theorem, proved by Kunen for $\kappa\ge\omega_2$ (see \cite[Lemma 5.2]{Kunen-special-points}).

\begin{cor}\label{cor:kunen} In $V^{\mathbb{M}_\kappa}$ no P-point extends a semi-selective ultrafilter from the ground model.
\end{cor}
\begin{proof} 
	Let $G$ be an $\mathbb{M}_\kappa$-generic filter over $V$. Let $\mathcal{U}$ be a semi-selective ultrafilter in $V$. Suppose that, in $V[G]$, an ultrafilter $\mathcal{W}$ extends $\mathcal{U}$. It means that $\mathcal{W} \in K_\mathcal{U}$.  By Corollary \ref{non-atomic} and Theorem \ref{semi-selective}, $K_\mathcal{U}$ supports a
	non-atomic measure. So, according to Proposition \ref{Ppoint-measure}, $\mathcal{W}$ cannot be a P-point. 
\end{proof}

\subsection{Canjar ultrafilters\label{sec:canjar}} The results of Section \ref{semi-sele} say that if we want to look for a P-point in the random model which is an extension of a ground model ultrafilter, then we have to avoid semi-selective ultrafilters. In particular,
whenever
$\mathcal{U}$ is a ground model semi-selective ultrafilter, then there is a decreasing sequence $(E_n)$ of infinite subsets of $\omega$ (in the extension) such that $\mathcal{U}\cup\{E_n\colon n\in \omega\}$ generates a filter which cannot be
further extended to a bigger filter by \emph{any} pseudointersection of $(E_n)$. (It is enough to take $\dot{E}_n$'s such that $\dot{\mu}_{\mathcal{U}}(\dot{E}_n)<1/(n+1)$ for each $n\in\omega$ (cf. Lemma \ref{Ms_measure}) and then use Theorem \ref{semi-selective}). In this
section we will show that such a sequence $(E_n)$ cannot exist for any Canjar ultrafilter $\mathcal{U}$. Therefore, perhaps the ground model Canjar ultrafilters have chances to be 'extendable' to a P-point in the random extension.

For an ultrafilter $\mathcal{U}$ and a condition $p\in\mathbb{M}_\kappa$, say that $M\in \mathbb{M}^\omega_\kappa$ is $(\mathcal{U},p)$-\emph{full} if for each $X\in \mathcal{U}$ we have
\[\lambda_\kappa\big(\bigvee_{k\in X} M(k)\wedge p\big)=\lambda_\kappa(p).\]
Equivalently, $M$ is $(\mathcal{U},p)$-full if and only if $\bigvee_{k\in X} M(k)\ge p$ for every $X\in\mathcal{U}$.

\begin{prop}\label{U-full} Assume that $\mathcal{U}$ is an ultrafilter on $\omega$ and $p\in\mathbb{M}_\kappa$ is a condition. Let $M\in \mathbb{M}_\kappa^\omega$. Then, 
	$M$ is $(\mathcal{U},p)$-full if and only if $p\Vdash$``$ \mathcal{U} \cup \{\dot{M}\}$ forms a base of a filter''.
\end{prop}
\begin{proof} Suppose first that $M$ is $(\mathcal{U},p)$-full but $p\nVdash$``$\mathcal{U} \cup \{\dot{M}\}$ forms a base of a filter''. It means that there is $q\le p$ and $X \in \mathcal{U}$ such that $q\Vdash X \cap \dot{M} = \emptyset$. But
	then $\bigvee_{k\in X} M(k) \land q = 0$ and so $M$ is not $(\mathcal{U},p)$-full.

	Conversely, suppose that $M$ is not $(\mathcal{U},p)$-full. Then, let $X\in \mathcal{U}$ be such that for $q = \bigvee_{k\in X} M(k)\wedge p$ we have $q<p$. Then, $p\setminus q \Vdash \dot{M} \cap X = \emptyset$, and so $p\setminus q \Vdash$ ``there is no filter containing $\mathcal{U}$ and $\dot{M}$''.
\end{proof}

In the following definition we treat $\mathcal{U}$ as a subset of $2^\omega$. Recall that a sequence $(I_n)$ of finite subsets of $\omega$ is \emph{an interval partition} if $\omega=\bigcup_{n\in\omega}I_n$ and $\max I_n=\min I_{n+1}-1$ for every $n\in\omega$.

\begin{df} An ultrafilter $\mathcal{U}$ is called \emph{Canjar}\footnote{The original definition of Canjar ultrafilters is different but equivalent to the following one (cf. \cite{Verner}).} if for every sequence $(\mathcal{C}_n)$ of compact subsets of $\mathcal{U}$ there is an interval partition $(I_n)$ of $\omega$ such that for every sequence $(X_n)$, where $X_n \in \mathcal{C}_n$ for every $n\in\omega$, we have
	\[ X = \bigcup_{n\in\omega}(X_n \cap I_n) \in \mathcal{U}. \]
\end{df}
Note that every Canjar ultrafilter is a P-point: if $(X_n)$ is a decreasing sequence in $\mathcal{U}$, then $X$ defined as above for the sequence $\mathcal{C}_n=\{X_n\}$ is a pseudointersection of $(X_n)$. The above definition is however stronger than that of a P-point---it says that we can diagonalize \emph{many} sequences of elements of an ultrafilter at
the same time.

The hope that Canjar ultrafilters may have an extension to a P-point in the random model comes from the following result and its consequence, Theorem \ref{almost-P-space}.

\begin{prop}\label{lem:canjar} Let $\mathcal{U}$ be a Canjar ultrafilter and fix a condition $p\in\mathbb{M}_\kappa$. Suppose that $(E_n)$ is a decreasing sequence of elements of $\mathbb{M}_\kappa^\omega$ which are $(\mathcal{U},p)$-full. Then, there is $(\mathcal{U},p)$-full $E\in \mathbb{M}_\kappa^\omega$ such that
	$E\leq^* E_n$ for each $n\in\omega$. 
\end{prop}
\begin{proof} For every $n\in\omega$ put 
	\[ \mathcal{C}_n = \Big\{X\in \mathcal{U}\colon\ \lambda_\kappa \Big(\bigvee_{k\in \omega\setminus X} E_n(k)\wedge p\Big)  \leq \lambda_\kappa(p)-1/(n+1)\Big\}. \]
			Notice that the set $\mathcal{C}_n$ is a compact subset of $\mathcal{U}$ such that $\omega\in\mathcal{C}_n$. Indeed, suppose that $X\in\wp(\omega)\setminus\mathcal{C}_n$. If $X\notin \mathcal{U}$, then $\omega\setminus X \in \mathcal{U}$ and so \[\lambda_\kappa
			\Big(\bigvee_{k\in \omega\setminus X} E_n(k)\wedge p\Big) = \lambda_\kappa(p),\] as $E_n$ is $(\mathcal{U},p)$-full. If $X\in \mathcal{U}\setminus \mathcal{C}_n$, then \[ \lambda_\kappa \Big(\bigvee_{k\in \omega\setminus X} E_n(k)\wedge p\Big)>\lambda_\kappa(p)-1/(n+1).\]
			In any case, there is $N\in \omega$ such that 
		\[ \lambda_\kappa \Big(\bigvee_{k\in N\setminus X} E_n(k)\wedge p\Big) > \lambda_\kappa(p)-1/(n+1). \]
		Then, $\{Y\subseteq \omega\colon Y\cap N = X\cap N\}$ is an open neighbourhood of $X$ disjoint with $\mathcal{C}_n$. It follows that $\mathcal{C}_n$ is a closed subset of $2^\omega$, hence it is compact.

	As $\mathcal{U}$ is Canjar, there is an interval partition $(I_n)$ such that, for every $(X_n)$ with $X_n \in \mathcal{C}_n$ for each $n\in\omega$, we have $\bigcup_{n\in\omega}(X_n \cap I_n)\in \mathcal{U}$.
	For every $n\in\omega$ and $k\in I_n$ set \[E(k)=E_n(k).\] 
	Clearly, $E\leq ^* E_n$ for each $n\in\omega$. We are going to show that $E$ is $(\mathcal{U},p)$-full. 

 	Suppose that $Y\subseteq \omega$ and $n_0\in \omega$ are such that 
	\[ \lambda_\kappa \Big(\bigvee_{k\in Y} E(k)\wedge p\Big) \leq \lambda_\kappa(p)-1/(n_0+1) < \lambda_\kappa(p). \]
	Let $X_n = (I_n \cap Y)^c$ for each $n\in\omega$. Since $X_n^c$ is finite, $X_n\in\mathcal{U}$. Then, for every $n>n_0$ we have:
	\[ \lambda_\kappa \Big(\bigvee_{k\in\omega\setminus X_n} E_n(k)\wedge p\Big) = \lambda_\kappa \Big(\bigvee_{k\in I_n \cap Y} E_n (k)\wedge p\Big) =\] \[= \lambda_\kappa\Big(\bigvee_{k\in I_n \cap Y} E(k)\wedge p\Big) \leq \lambda_\kappa\Big(\bigvee_{k\in Y} E(k)\wedge p\Big) \leq \lambda_\kappa(p)-1/(n+1), \]
 so $X_n \in \mathcal{C}_n$.
	Hence, $X = \bigcup_{n>n_0} (X_n \cap I_n) \in \mathcal{U}$. But $X \cap Y = \emptyset$ and so $Y\notin \mathcal{U}$, and we are done.
\end{proof}

Recall that a topological space is \emph{an almost P-space} if its every non-empty $\mathbb{G}_\delta$ subset has non-empty interior. 

\begin{thm}\label{almost-P-space} Suppose $\mathcal{U}$ is a Canjar ultrafilter in $V$. Then, $\Vdash_{\mathbb{M}_\kappa}$``$\dot{K}_\mathcal{U}$ is an almost P-space''. 
\end{thm}

\begin{proof}
Suppose that $p\in\mathbb{M}_\kappa$ is a condition and $\dot{A}$ is an $\mathbb{M}_\kappa$-name such that $p\Vdash$``$\dot{A}$ is a non-empty $\mathbb{G}_\delta$ subset of $\dot{K}_\mathcal{U}$''. Then, by the Maximum Principle and Lemma \ref{restr_incl}, there is a decreasing sequence $(E_n)$ of elements of $\mathbb{M}^\omega_\kappa$ such that \[p\Vdash\emptyset \ne \bigcap_n
	[\dot{E}_n] \cap \dot{K}_\mathcal{U} \subseteq \dot{A}.\]

	Since for each $n\in\omega$ we have $p\Vdash [\dot{E}_n] \cap \dot{K}_\mathcal{U} \ne \emptyset$, Proposition \ref{U-full} implies that each $E_n$ is $(\mathcal{U},p)$-full. By Proposition \ref{lem:canjar} there is a $(\mathcal{U},p)$-full $E\in \mathbb{M}^\omega_\kappa$
	such that $E\le^* E_n$ for every $n\in\omega$. Then, again by Proposition \ref{U-full},
\[ p\Vdash[\dot{E}] \cap \dot{K}_\mathcal{U} \ne \emptyset \mbox{ and }\dot{E} \subseteq^* \dot{E_n} \mbox{ for each }n\in\omega,\] and so \[ p\Vdash\emptyset \ne
		[\dot{E}] \cap \dot{K}_\mathcal{U}
	\subseteq \bigcap_n [\dot{E}_n] \cap \dot{K}_\mathcal{U} \subseteq \dot{A}.\]
Since $p$ and $\dot{A}$ were arbitrary, we are done.
\end{proof}

\bibliographystyle{alpha}
\bibliography{bib-ppoint}

\end{document}